\documentclass[a4paper,1pt]{amsart}
\usepackage{graphicx}
\usepackage{amssymb}
\usepackage{amsmath}
\usepackage{latexsym}
\usepackage{amsthm}
\usepackage{autograph,epic,latexsym,bezier,amsbsy,psfrag,color,enumerate,amsfonts,amsmath,amscd,amssymb}
\usepackage[latin1]{inputenc}

\setloopdiam{17}\setprofcurve{10}
\newtheorem{defi}{Definition}[section]
\newtheorem{teo}[defi]{Theorem}
\newtheorem{lem}[defi]{Lemma}

\newtheorem{os}[defi]{Remark}

\begin{document}

\title[Horofunctions on Sierpi\'{n}ski type triangles]{Horofunctions on Sierpi\'{n}ski type triangles}

\author{Daniele D'Angeli}\thanks{The author was supported by Austrian Science Fund projects FWF
P24028-N18 and P29355-N35.}
\address{Institut f\"{u}r Diskrete Mathematik\\
Technische Universit\"{a}t Graz \ \ Steyrergasse 30, 8010 Graz, Austria}
\email{dangeli@math.tugraz.at}


\date{\today}

\begin{abstract}
We study an infinite set of graphs which are recursively
constructed from an infinite word in a finite alphabet. These
graphs are inspired by the construction of the Sierpi\'{n}ski
gasket. We show that there are infinitely many non-isomorphic such
graphs and we describe the horofunctions on the standard case.
\end{abstract}

\maketitle

\section{Introduction}
This papers deals with two combinatorial aspects related to the
so-called Sierpi\'{n}ski gasket. This  graph belongs to the class of
fractal objects of noninteger Hausdorff dimension. It is realized
by a repeated construction of an elementary shape on progressively
smaller length scales. The Sierpi\'{n}ski gasket appears in
different contexts: analysis on fractals \cite{analysis}, \cite{woess}, \cite{tep} some physical models as dimer and Ising
models \cite{taiwan}, \cite{ddn}, \cite{noidimeri}, \cite{wagner1}
and combinatorics \cite{alberi}, \cite{tutte}. Moreover the Sierpi\'{n}ski gasket
is the limit space of the Hanoi Towers group on three pegs
\cite{hanoi} establishing a connection with the theory of
self-similar groups \cite{volo}. One can construct an infinite
sequence of finite graphs which are inspired to the Sierpi\'{n}ski
gasket. In this paper we inductively construct such a sequence
getting a natural limit graph which is an infinite marked graph.
The vertices of this graph are labelled by infinite words over a
finite alphabet of three symbols. This coding has a geometrical
meaning: by reading the $n-$th letter of such infinite word we can
construct the next (finite) graph of the finite sequence, and mark it in a
precise vertex (see \cite{carpet} for comparison). The results contained in the paper are the
following: first, we give a classification, up to isomorphisms of
such infinite graphs depending on the corresponding infinite
words; then we study the horofunction of a particular case that we call \textit{standard}. The
problem of isomorphism is studied in the context of non-marked
graphs, i.e., we construct such infinite marked graphs and then we
forget the marked vertex and compare them. The study of
horofunctions is a classical topic in the setting of
$C^{\ast}$-algebras and Cayley graphs of groups, giving rise to
the description of the Cayley compactification and the boundary of
a group \cite{devin}, \cite{rief}. It is inspired by the seminal work of Gromov in the study of a suitable definition of boundary for a matric space \cite{gromov}. Given a proper matric space $X$, one associates with every point of $X$ a continuous real-valued function in the space endowed with the topology of uniform convergence on compact sets. The topological boundary of this space of functions modulo the constant functions is called the horofunction boundary of $X$ \cite{walsh}. It is immediate that in the case of an infinite graph $G$ with the usual metric $d$ one gets the discrete topology, so every function on it is continuous. This implies that $(G, d)$ is automatically complete and locally compact.
Moreover in our case $G$ is countable, and $d$ is
proper since every vertex has finite degree.
The fractal structure of the graphs studied in this paper allows to give a complete description of the horofunctions.

\section{Sierpi\'{n}ski type triangles}
In this section we define infinitely many marked graphs
that can be thought as approximations of the famous Sierpi\'{n}ski
gasket. Each of these graphs is inductively constructed from an infinite word in a finite alphabet. The vertices of such graphs are labeled by infinite words in this alphabet.

We will show
that there are infinitely many isomorphism classes of such graphs, regarded as non-marked graphs.

Let us start by fixing the finite alphabet $X=\{u,l,r\}$, and a
triangle with the vertices labeled by $u$ (up), $l$ (left), $r$
(right), with the obvious geometric meaning. Take an
infinite word $w=w_1w_2\cdots$, with $w_i\in X$. We denote by
$\underline{w}_n$ the prefix $w_1\cdots w_n$ of length $n$ of $w$.
\begin{defi}
The infinite Sierpi\'{n}ski type graph $\Gamma_{w}$ is the marked
graph inductively constructed as follows:
\begin{description}
  \item[Step 1] Mark the vertex corresponding to $w_1$ in the simple
  triangle. Denote this marked graph $\Gamma_{w}^1$.
  \item[Step 2] Take three copies of $\Gamma_{w}^n$ and glue them
  together in such a way that each one shares exactly one (extremal) vertex with
  each other copy. These copies occupy the up, left or right position in the new graph. This has, by construction, three marked points, we keep
  the one in the copy corresponding to the letter $w_{n+1}$. Call this graph $\Gamma_{w}^{n+1}$ and identify with $\underline{w}_{n+1}$ the marked vertex.
  \item[Step 3] Denote by $\Gamma_{w}$ the limit of the marked graphs
  $\Gamma_{w}^n$.
\end{description}
\end{defi}

The limit in the previous definition means that in $\Gamma_w$ we have an increasing sequence of subgraphs marked at $\underline{w}_n$ isomorphic to the graphs $\Gamma_{w}^n$.

By definition, for each $w$, the graph $\Gamma_w$ is the limit of
finite graphs $\Gamma_w^n$, $n\geq 1$. Each one of these finite
graphs has three external points, that can thought as the
\textit{boundary} of the graph $\Gamma_w^n$. These represent the points where we (possibly) glue the three copies of $\Gamma_w^n$ to get $\Gamma_w^{n+1}$. We denote them by
$U_n^w=u\cdots u$ (the upmost one), $L_n^w=l\cdots l$ (the leftmost one), $R_n^w=r\cdots r$ (the
rightmost one). More precisely, passing from $\Gamma_w^n$ to $\Gamma_w^{n+1}$, we identify $L_n^w$ and $R_n^w$ of the upper copy isomorphic to $\Gamma_w^n$ with the vertices $U_n^w$ of the left and right copies isomorphic to $\Gamma_w^n$, respectively, and we identify $L_n^w$ of the right copy isomorphic to $\Gamma_w^n$ with $R_n^w$ of the left one.

We label by the same letters the corresponding
vertices in $\Gamma_w$.

\begin{os}\rm
If we consider the word $w:=l^{\infty}$, then we get
\vspace{1cm}

\end{os}\unitlength=0,2mm
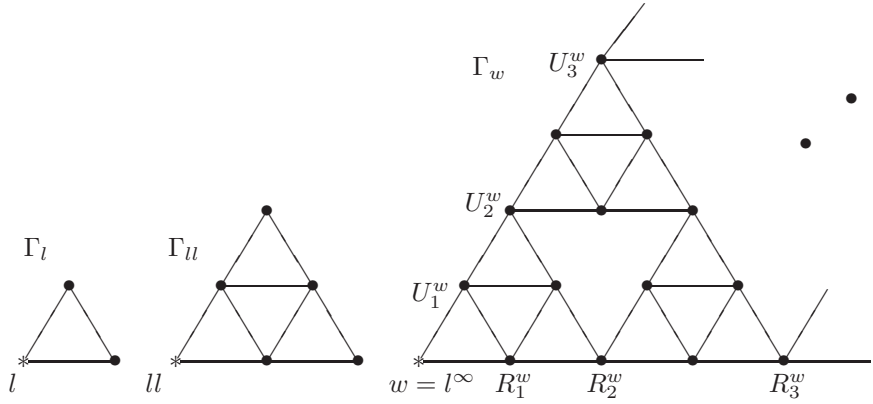
\begin{center}
\begin{picture}(500,220)
\letvertex a=(30,60)\letvertex b=(0,10)\letvertex c=(60,10)

\letvertex d=(160,110)\letvertex e=(130,60)\letvertex f=(100,10)\letvertex g=(160,10)

\letvertex h=(220,10)\letvertex i=(190,60)

\put(-10,-10){$l$} \put(80,-10){$ll$}
\put(240,-10){$w=l^{\infty}$}

\drawvertex(a){$\bullet$}\drawvertex(b){$\ast$}
\drawvertex(c){$\bullet$}\drawvertex(d){$\bullet$}
\drawvertex(e){$\bullet$}\drawvertex(f){$\ast$}
\drawvertex(g){$\bullet$}\drawvertex(h){$\bullet$}
\drawvertex(i){$\bullet$}

\drawundirectededge(b,a){} \drawundirectededge(c,b){}
\drawundirectededge(a,c){} \drawundirectededge(e,d){}
\drawundirectededge(f,e){} \drawundirectededge(g,f){}

\drawundirectededge(h,g){} \drawundirectededge(i,h){}
\drawundirectededge(d,i){} \drawundirectededge(i,e){}
\drawundirectededge(e,g){} \drawundirectededge(g,i){}
\put(0,80){$\Gamma_l$}

\put(295,200){$\Gamma_w$}

\put(95,80){$\Gamma_{ll}$}
\letvertex A=(380,210)\letvertex B=(350,160)\letvertex C=(320,110)

\letvertex D=(290,60)\letvertex E=(260,10)\letvertex F=(320,10)\letvertex G=(380,10)

\letvertex H=(440,10)\letvertex I=(500,10)
\letvertex L=(470,60)\letvertex M=(440,110)\letvertex N=(410,160)

\letvertex O=(380,110)\letvertex P=(350,60)\letvertex Q=(410,60)
\letvertex R=(450,210) \letvertex S=(560,10)
\letvertex T=(410,250)\letvertex U=(530,60)

\put(310,-10){$R^w_1$}\put(255,50){$U^w_1$}
\put(370,-10){$R^w_2$}\put(290,110){$U^w_2$}
\put(345,203){$U^w_3$}\put(490,-10){$R^w_3$}

\drawvertex(A){$\bullet$}\drawvertex(B){$\bullet$}
\drawvertex(C){$\bullet$}\drawvertex(D){$\bullet$}
\drawvertex(E){$\ast$}\drawvertex(F){$\bullet$}
\drawvertex(G){$\bullet$}\drawvertex(H){$\bullet$}
\drawvertex(I){$\bullet$}\drawvertex(L){$\bullet$}\drawvertex(M){$\bullet$}
\drawvertex(N){$\bullet$}\drawvertex(O){$\bullet$}
\drawvertex(P){$\bullet$}\drawvertex(Q){$\bullet$}

\drawundirectededge(A,R){}\drawundirectededge(A,T){}\drawundirectededge(I,S){}\drawundirectededge(I,U){}
\drawundirectededge(E,D){} \drawundirectededge(D,C){}
\drawundirectededge(C,B){} \drawundirectededge(B,A){}
\drawundirectededge(A,N){} \drawundirectededge(N,M){}
\drawundirectededge(M,L){} \drawundirectededge(L,I){}
\drawundirectededge(I,H){} \drawundirectededge(H,G){}
\drawundirectededge(G,F){} \drawundirectededge(F,E){}
\drawundirectededge(N,B){} \drawundirectededge(O,C){}
\drawundirectededge(M,O){} \drawundirectededge(P,D){}
\drawundirectededge(L,Q){} \drawundirectededge(B,O){}
\drawundirectededge(O,N){} \drawundirectededge(C,P){}
\drawundirectededge(P,G){} \drawundirectededge(D,F){}
\drawundirectededge(Q,M){}
\drawundirectededge(G,Q){}\drawundirectededge(H,L){}\drawundirectededge(F,P){}
\drawundirectededge(Q,H){}

\put(510,150){$\bullet$}\put(540,180){$\bullet$}
\end{picture}
\end{center}

%

We want to study the isomorphism problem for such graphs. We
consider these graphs as non-marked and denote by $``\simeq"$ the
corresponding equivalence relation. More explicitly, $G\simeq G'$ if there exists an isomorphism $\phi:G\rightarrow G'$. Notice that, if there exists such isomorphism $\phi$ (of two non-marked graphs), it can be seen as an isomorphism of the marked graphs $(G,v)$ and $(G',v')$ rooted at the points $v\in G$ and $v'=\phi(v)\in G'$, for any $v\in G$. We will use this observation in the sequel.

Moreover, it is clear that the limit
$\Gamma_{w}$ is obtained as an exhaustion of triangles. What $w$
detects is the position of the smaller triangles inside the bigger
ones. In what follows $d(\cdot,\cdot)$ will denote the discrete distance in the graphs regarded as metric spaces.

\begin{os}\rm
We stress the fact that, even if all (marked) graphs $\{\Gamma_w\}$ are limits of the same finite graphs $\{\Gamma^n_w\}$, they are a priori non isomorphic as non marked graphs. An easy example to show that consists in considering the graphs $\Gamma_{l^{\infty}}$ and $\Gamma_{v}$, where $v=v_1v_2\cdots$ and $v_i$ is not definitively equal to $u,l$ or $r$ (see the figure below). In this case $\Gamma_{v}$ does not contain any vertex of degree $2$, contrary to $\Gamma_{l^{\infty}}$. In particular there is no isomorphism between the two graphs.
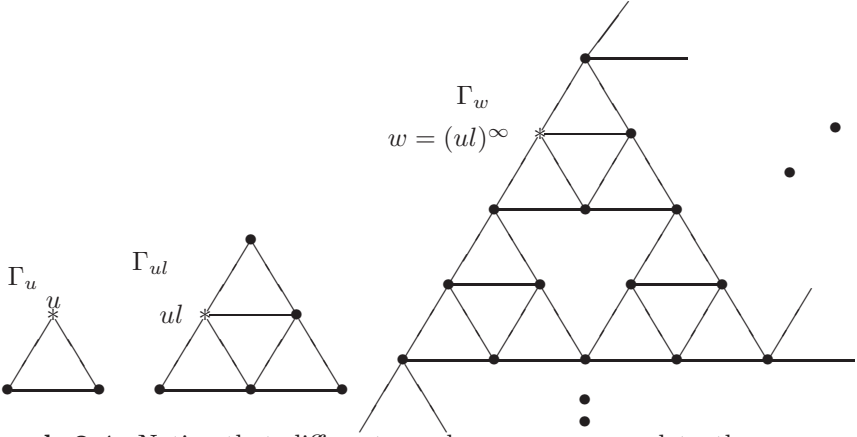
\vspace{1cm}
\begin{center}
\begin{picture}(500,220)
\letvertex a=(30,60)\letvertex b=(0,10)\letvertex c=(60,10)

\letvertex d=(160,110)\letvertex e=(130,60)\letvertex f=(100,10)\letvertex g=(160,10)

\letvertex h=(220,10)\letvertex i=(190,60)

\put(25,65){$u$} \put(100,55){$ul$}
\put(250,172){$w=(ul)^{\infty}$}

\drawvertex(a){$\ast$}\drawvertex(b){$\bullet$}
\drawvertex(c){$\bullet$}\drawvertex(d){$\bullet$}
\drawvertex(e){$\ast$}\drawvertex(f){$\bullet$}
\drawvertex(g){$\bullet$}\drawvertex(h){$\bullet$}
\drawvertex(i){$\bullet$}

\drawundirectededge(b,a){} \drawundirectededge(c,b){}
\drawundirectededge(a,c){} \drawundirectededge(e,d){}
\drawundirectededge(f,e){} \drawundirectededge(g,f){}

\drawundirectededge(h,g){} \drawundirectededge(i,h){}
\drawundirectededge(d,i){} \drawundirectededge(i,e){}
\drawundirectededge(e,g){} \drawundirectededge(g,i){}
\put(0,80){$\Gamma_u$}

\put(295,200){$\Gamma_w$}

\put(82,90){$\Gamma_{ul}$}
\letvertex A=(380,230)\letvertex B=(350,180)\letvertex C=(320,130)

\letvertex D=(290,80)\letvertex E=(260,30)\letvertex F=(320,30)\letvertex G=(380,30)

\letvertex H=(440,30)\letvertex I=(500,30)
\letvertex L=(470,80)\letvertex M=(440,130)\letvertex N=(410,180)

\letvertex O=(380,130)\letvertex P=(350,80)\letvertex Q=(410,80)
\letvertex R=(450,230) \letvertex S=(560,30)
\letvertex T=(410,270)\letvertex U=(530,80)

\letvertex AA=(290,-20)\letvertex BB=(230,-20)\letvertex CC=(470,-20)\letvertex DD=(530,-20)

\drawvertex(A){$\bullet$}\drawvertex(B){$\ast$}
\drawvertex(C){$\bullet$}\drawvertex(D){$\bullet$}
\drawvertex(E){$\bullet$}\drawvertex(F){$\bullet$}
\drawvertex(G){$\bullet$}\drawvertex(H){$\bullet$}
\drawvertex(I){$\bullet$}\drawvertex(L){$\bullet$}\drawvertex(M){$\bullet$}
\drawvertex(N){$\bullet$}\drawvertex(O){$\bullet$}
\drawvertex(P){$\bullet$}\drawvertex(Q){$\bullet$}

\drawundirectededge(A,R){}\drawundirectededge(A,T){}\drawundirectededge(I,S){}\drawundirectededge(I,U){}
\drawundirectededge(E,D){} \drawundirectededge(D,C){}
\drawundirectededge(C,B){} \drawundirectededge(B,A){}
\drawundirectededge(A,N){} \drawundirectededge(N,M){}
\drawundirectededge(M,L){} \drawundirectededge(L,I){}
\drawundirectededge(I,H){} \drawundirectededge(H,G){}
\drawundirectededge(G,F){} \drawundirectededge(F,E){}
\drawundirectededge(N,B){} \drawundirectededge(O,C){}
\drawundirectededge(M,O){} \drawundirectededge(P,D){}
\drawundirectededge(L,Q){} \drawundirectededge(B,O){}
\drawundirectededge(O,N){} \drawundirectededge(C,P){}
\drawundirectededge(P,G){} \drawundirectededge(D,F){}
\drawundirectededge(Q,M){}
\drawundirectededge(G,Q){}\drawundirectededge(H,L){}\drawundirectededge(F,P){}
\drawundirectededge(Q,H){}
\drawundirectededge(AA,E){}\drawundirectededge(BB,E){}

\put(375,0){$\bullet$}\put(375,-15){$\bullet$}
\put(510,150){$\bullet$}\put(540,180){$\bullet$}
\end{picture}
\end{center}
\end{os}

\begin{os}\label{stesse}\rm Notice that different words may correspond to the same
vertex of $\Gamma_w^n$. More precisely for any $k$ and $n>k+1$ the following
pairs of vertices are identified $u^klv=l^kuv$,$u^krv=r^kuv$ and
$r^klv=l^krv$, for any word $v$ in the alphabet $X$. From now on we consider such elements as identified and choose just one representation for them.
\end{os}

Recall that two infinite words $w$ and $v$ are cofinal if there
exists $n$ in $\mathbb{N}$ such that $v_i= w_i$ for all $i>n$.
This is clearly an equivalence relation:  the \textit{cofinality}
and we denote it by $``\sim"$.\\

Notice that an infinite word $x$ corresponds to a vertex of $\Gamma_v$ if and only if $x\sim v$ (see Remark \ref{cofi}).

In what follows, given an infinite word $x\in \Gamma_v$ we use the notation $x\in \Gamma_v^n$ meaning that $x$ is an infinite word corresponding to a vertex belonging to the $n-$th subgraph $\Gamma_v^n$ of $ \Gamma_v$, obtained after the first $n$ steps in the construction of $\Gamma_v$. More precisely, $x\in \Gamma_v^n$ if $x\in \Gamma_v$ and it corresponds to a vertex of the subgraph $\Gamma_v^n\hookrightarrow \Gamma_v$ in the natural embedding of the finite graph $\Gamma_v^n$ into the infinite graph $\Gamma_v$.

If we want to emphasize the vertices of the finite graph $\Gamma_v^n$ we prefer using the notation $\underline{x}_n$.

\begin{lem}\label{lemma0}
If $v\sim w$ then $\Gamma_v\simeq \Gamma_w$.
\end{lem}
\begin{proof}
If $v\sim w$ then there is $n$ such that $v_{n+k}=w_{n+k}$ for all
$k\geq 1$. The graphs $\Gamma_v^n$ and $\Gamma_w^n$ are isomorphic
as non-marked graphs by construction. Let
$\phi_n:\Gamma_v^n\rightarrow \Gamma_w^n$ be the identity isomorphism
of non-marked graphs. Then $\phi_n$ extends to an isomorphism
$\phi :\Gamma_v\rightarrow\Gamma_w$ since $v_{n+k}=w_{n+k}$ for
each $k\geq 1$.
\end{proof}

\begin{lem}\label{lemma1}
$\Gamma_v\simeq \Gamma_w$ if and only if there exist two vertices
$x\in\Gamma_v$ and $y\in \Gamma_w$ such that, for every $n$ the
sets 
$$
\{d(x,U_n^x),d(x,L_n^x), d(x,R_n^x)\}
$$ 
and
$$
\{d(y,U_n^y),d(y,L_n^y), d(y,R_n^y)\}
$$ 
coincide.
\end{lem}
\begin{proof}
First notice that, since $x\in \Gamma_v$ and $y\in \Gamma_w$ then $x\sim v$ and $y\sim w$ so that for $T=U,L,R$ one has $T_n^x=T_n^v$ and $T_n^y=T_n^w$ for every $n$ sufficiently large.

Suppose there exists an isomorphism $\phi:\Gamma_v\rightarrow
\Gamma_w$ such that $\phi(x)=y$. And assume there is $n$ such that
$$
\{d(x,U_n^x),d(x,L_n^x), d(x,R_n^x)\}\neq
\{d(y,U_n^y),d(y,L_n^y), d(y,R_n^y)\}.
$$ 
Denote by
$M_n:=\min\{m_x^n,m_y^n\}$ where
$$
m_x^n:=\max\{d(x,U_n^x),d(x,L_n^x), d(x,R_n^x)\}
$$ 
and
$$
m_y^n:=\max\{d(y,U_n^y),d(y,L_n^y), d(y,R_n^y)\}.$$ 
We claim that
the balls\\
$B_x(M_n)$ and $B_y(M_n)$ are not isomorphic, as graphs.
If $m_x^n\neq m_y^n$ then only one of $B_x(M_n)$ and $B_y(M_n)$
contains a copy isomorphic to $\Gamma_x^n$ (regarded as a non marked
graph). If $m_x^n= m_y^n$ then, the other two distances from $x$
and $y$ to the boundary vertices do not coincide and so the part
of the graphs $B_x(M_n)$ and $B_y(M_n)$ exceeding the copy of $\Gamma^n_x$ and $\Gamma^n_y$, respectively are not
isomorphic.

Viceversa suppose that the sets $\{d(x,U_n^x),d(x,L_n^x),
d(x,R_n^x)\}$ and\\ 
$\{d(y,U_n^y),d(y,L_n^y), d(y,R_n^y)\}$ coincide
for each $n$.
Let $\phi:\Gamma_v\rightarrow \Gamma_w$ be the map such that $\phi(x)=y$.
The balls $B_x(M_n)$ and $B_y(M_n)$ are
isomorphic for each $n$ and so the map $\phi$ is an isomorphism of (non-marked) graphs, since $\lim B_x(M_n)=\Gamma_v$ and $\lim
B_y(M_n)=\Gamma_w$ regarded as non-marked graphs.
\end{proof}

\begin{lem}\label{lemmino}
Let $\underline{x}_n, \underline{y}_n \in \Gamma_w^n$ be two
vertices such that\\
$\{d(\underline{x}_n,U_n^w),d(\underline{x}_n,L_n^w),
d(\underline{x}_n,R_n^w)\}=$ $
\{d(\underline{y}_n,U_n^w),d(\underline{y}_n,L_n^w),
d(\underline{y}_n,R_n^w)\}$\\
then the same holds for every $k\leq n$.
\end{lem}
\begin{proof}
Observe that, in general, when we pass from $\Gamma_s^n$ to $\Gamma_s^{n+1}$ exactly one of the elements in $\{d(x,U_n^s),d(x,L_n^s),
d(x,R_n^s)\}$ is preserved. More precisely if $s_n=t \in \{u,l,r\}$ then $d(x,T_n^s)=d(x,T_{n+1}^s)$, for $T\in \{U,L,R\}$. In the other cases the distance increases by $2^{n-1}$. This implies that if there is a $k$ in which the sets of distances do not coincide then they cannot coincide for $k+1$.
\end{proof}

The following result describes points with same distances from the
boundary points, as elements in the same orbit under the action of
the symmetric group. We must take into account the exceptions of the Remark (\ref{stesse}).

\begin{lem}\label{lemma1.1}
Let $\underline{x}_n,\underline{y}_n \in \Gamma_w^n$ be two
vertices. Then \\
$\{d(\underline{x}_n,U_n^w),d(\underline{x}_n,L_n^w),
d(\underline{x}_n,R_n^w)\}=$ $
\{d(\underline{y}_n,U_n^w),d(\underline{y}_n,L_n^w),
d(\underline{y}_n,R_n^w)\}$ if and only if there exists $\sigma\in
Sym(\{u,l,r\})$ such that
$\sigma(\underline{x}_n)=\underline{y}_n$, where
$$
\sigma(\underline{x}_n)=\sigma(x_1)\cdots\sigma(x_{n-1})\sigma(x_n)=\sigma(\underline{x}_{n-1})\sigma(x_n),
$$.
\end{lem}
\begin{proof}
Suppose the sets of distances coincide and proceed by induction on $n$. For $n=1$ the assertion is trivially verified.
First suppose, without loss of generality, that $x_n=y_n=u$. It follows from Lemma (\ref{lemmino}) that there exists a permutation $\sigma\in
Sym(\{u,l,r\})$ such that
$\sigma(\underline{x}_{n-1})=\underline{y}_{n-1}$. We want to prove that this permutation is the identity or the transposition $(l,r)$.
The distances of $\underline{x}_{n-1}$ and $\underline{y}_{n-1}$ from $U_{n-1}^w$ must be the same. This implies that $\{i: \ x_i=u\}=\{i: \ y_i=u\}$. For the indices which are not equal to $u$ we observe that, $d(\underline{x}_{n-1},R_{n-1}^w)$ is equal either to $d(\underline{y}_{n-1},R_{n-1}^w)$ or to $d(\underline{y}_{n-1},L_{n-1}^w)$. The first case gives the identity, the second case the transposition $(l,r)$.

Suppose now that $x_n\neq y_n$, this implies that $\underline{x}_n$ and $\underline{y}_n$ belong to different copies isomorphic to $\Gamma_w^{n-1}$ of the graph $\Gamma_w^n$. Suppose, for example that $x_n=u$ and $y_n=r$. If $d(\underline{x}_{n},L_{n}^w)=d(\underline{y}_{n},L_{n}^w)$ as before we can show that $\sigma$ is equal to $(r,u)$, since we have $d(\underline{x}_{n-1},L_{n-1}^w)=d(\underline{y}_{n-1},L_{n-1}^w)$. If the distances of $\underline{x}_{n}$ and $\underline{y}_{n}$ from one (all) of the boundary vertices do not coincide, we have $d(\underline{x}_{n},U_{n}^w)=d(\underline{y}_{n},R_{n}^w)$ and $d(\underline{x}_{n},L_{n}^w)=d(\underline{y}_{n},U_{n}^w)$, $d(\underline{x}_{n},R_{n}^w)=d(\underline{y}_{n},L_{n}^w)$. The same property holds at level $n-1$ so that there exists a permutation $\sigma$ such that $\sigma(\underline{x}_{n-1})=\underline{y}_{n-1}$. This permutation cannot be the transposition $(r,u)$ because this would imply that $d(\underline{x}_{n-1},L_{n-1}^w)=d(\underline{y}_{n-1},
L_{n-1}^w)$. And so it is the permutation $(u,r,l)$.

On the other hand we prove that permutations preserve distances from boundary in the case that $\sigma$ is a transposition, say
$\sigma=(l,r)$. One can easily verify that the same argument can be applied to any permutation of $Sym(\{u,l,r\})$.
If $n=1$ it is clear that $x_1$ and $\sigma(x_1)$
satisfy the claim. It follows that
$$
\sigma(\underline{x}_n)=\sigma(x_1)\cdots\sigma(x_{n-1})\sigma(x_n)=\sigma(\underline{x}_{n-1})\sigma(x_n),
$$
and the sets of distances from $\underline{x}_{n-1}$ and
$\sigma(\underline{x}_{n-1})$ to the boundary points coincide by
induction. These two points belong to the graph $\Gamma_w^{n-1}$,
and one is obtained by the other via the transformation
$(l,r)$ corresponding to the reflection with respect to the
vertical axis. If $x_n=u$ then both $\underline{x}_{n}$ and
$\sigma(\underline{x}_{n})$ are vertices of the upper part of
$\Gamma_w^n$ obtained from each other by the same reflection. If
$x_n=l$ (resp. $x_n=r$) the vertices $\underline{x}_{n}$ and
$\sigma(\underline{x}_{n})$ live, respectively, in the left and
right (resp. right and left) part of $\Gamma_w^n$, and so they are
obtained from each other by the reflection with respect to the
vertical axis, and in particular preserve distances to the
boundary vertices.
\end{proof}

The group $Sym(\{u,l,r\})$ consists of six elements and its action
factorizes in orbits consisting of three (e.g. the boundary
points) or six elements.

\begin{teo}\label{teo1}
There are infinitely many isomorphism classes of the graphs
$\Gamma_w$. More precisely $\Gamma_v\simeq \Gamma_w$ if and only there exists $\sigma \in Sym(\{u,l,r\})$ such that $w\sim \sigma(v)$.
\end{teo}
\begin{proof}
Suppose that there exists $\sigma \in Sym(\{u,l,r\})$ such that $\sigma(v)\sim w$, hence there is $N\in \mathbb{N}$ and there is
$\sigma\in \ Sym(\{u,l,r\})$ such that $\forall n\geq N$ one has
$\sigma(v_n)=w_n$. Consider the graphs $\Gamma_v$ and $\Gamma_{\sigma(v)}$ and the sequences of increasing subgraphs $\{\Gamma_v^n\}$ and $\{\Gamma_{\sigma(v)}^n\}$. From Lemma (\ref{lemma1.1}) we get
$$
\{d(v,U_n^v),d(v,L_n^v), d(v,R_n^v)\}\!=\!\{d(\sigma(v),U_n^{\sigma(v)}),d(\sigma(v),L_n^{\sigma(v)}),\! d(\sigma(v),R_n^{\sigma(v)})\}
$$
for every $n$. Hence Lemma (\ref{lemma1}) implies that $\Gamma_v$ and $\Gamma_{\sigma(v)}$ are isomorphic. Lemma (\ref{lemma0}) gives that $\Gamma_v \simeq \Gamma_w$.

Viceversa suppose there exists an isomorphism
$\phi:\Gamma_v\rightarrow \Gamma_w$. Without loss of generality we can assume that $\phi(v)=w$. If we restrict $\phi$ to the
finite graphs $\Gamma_v^n$ and $\Gamma_w^n$ we get, by Lemma
(\ref{lemma1.1}), an isomorphism $\phi_n$ which corresponds to a
permutation $\sigma \in Sym(\{u,l,r\})$. This automorphism is determined by the images
of the boundary vertices $U_n^v,L_n^v,R_n^v$. Only one of them
will coincide with the corresponding boundary vertex of the graph
$\Gamma_v^{n+1}$, the same for $U_n^w,L_n^w,R_n^w$. More precisely, if $T\in \{U,L,R\}$ is such that $T_n^v=T_{n+1}^v$ then
$\phi(T)_{n}^w=\phi(T_n^v)=\phi(T_{n+1}^v)=\phi(T)_{n+1}^w$, for $T=U,L$ or $R$. So the
the permutation yielding the isomorphism between the graphs $\Gamma_v^{n+1}$ and $\Gamma_w^{n+1}$ is given by
the same permutation $\sigma$. This implies that $\phi(v)=\sigma(v)=w$. The cofinality
follows.
\end{proof}

\begin{os}\label{cofi}\rm
We can refine the last statement by considering that each
infinite word $v$ cofinal with $w$ can be seen as a vertex of the
graph $\Gamma_w$. In fact if it belongs to the same graph of $w$,
there exists $n$ such that $\underline{v}_n$ and $\underline{w}_n$
are vertices in $\Gamma_w^n$. This implies that $v_{n+k}=w_{n+k}$
for each $k\geq 1$. So Theorem (\ref{teo1}) implies that each
isomorphism class contains exactly $6$ graphs, except the class
of constant words which contains only $3$ graphs (because the orbit
of such a word under $Sym(\{u,l,r\})$ contains only three
elements)
\end{os}

\section{Horofunctions}

In this section we explicitly compute the horofunctions for the
graphs constructed in the previous section. We need some
definitions (for more details see \cite{devin} or \cite{walsh}).

Let $G=(V,E)$ be a graph, and
$\{x_n\}_{n\in \mathbb{N}}$ be a sequence of vertices such that
$d(o, x_n)\rightarrow\infty$. For every $n$ we define the function
$$
f_n(y):=d(x_n,o)-d(x_n,y),
$$
whose limit for $n\to \infty$, considered in the space of (continuous) functions on $G$ with the topology of uniform convergence on finite sets, gives the horofunction associated with the sequence $\{x_n\}_{n\in \mathbb{N}}$.

One considers the space of horofunctions up to the equivalence
relation which identifies functions whose difference is uniformly
bounded. Points which are limit of the geodesic rays in the graph
are called \emph{Busemann points} (see \cite{Cormann}). The notion of Busemann point was introduced by Rieffel in \cite{rief}.

We want to study the horofunctions on the Sierpi\'{n}ski type graph corresponding to $w=l^{\infty}$ up to the
equivalence stated above.

\subsection{The standard case}
In this section we compute the horofunctions of the infinite graph
$\Gamma_w$ where $w=l^{\infty}$. We choose $o=w$. Let $\{x_n\}$ be
a sequence of vertices in $\Gamma_w$ such that
$d(o,x_n)\rightarrow\infty$. By construction, for each $n$ there
exists $k=k(n)$ such that $x_n\in \Gamma_w^k\setminus
\Gamma_w^{k-1}$. The following result gives a necessary condition
for the existence of the limit of the functions $f_n$. In what
follows we omit the superscript $w$.

\begin{lem}\label{lemma2}
Suppose that there exist infinitely many indices $i$ such that
$d(x_i,U_{k(i)-1}) < d(x_i,R_{k(i)-1})$ and infinitely many
indices $j$ such that 
$$
d(x_j,U_{k(j)-1})>d(x_j,R_{k(j)-1}).
$$ 
Then $\lim f_n$ does not exist.
\end{lem}
\begin{proof}
Take $y$ such that $d(o,y)=1$, for example $y=U_1$. For the
indices $i$ such that $d(x_i,U_{k(i)-1})<d(x_i,R_{k(i)-1})$ we
have $d(x_i,o)=d(x_i,U_{k(i)-1})+d(o,U_{k(i)-1})$ and
$d(x_i,U_1)=d(x_i,U_{k(i)-1})+d(U_1,U_{k(i)-1})$ and so
$f_i(U_1)=1$. Analogously one can prove that $f_j(U_1)=0$. And so
the limit of $f_n(U_1)$ does not exist.
\end{proof}

The previous Lemma implies that the sequences $\{x_n\}$ that we have
to consider to have a limit are those such that $d(x_i,U_{k(i)-1})<d(x_i,R_{k(i)-1})$, $d(x_i,U_{k(i)-1})>d(x_i,R_{k(i)-1})$ or
$d(x_i,U_{k(i)-1})=d(x_i,R_{k(i)-1})$, provided $i$ is sufficiently large.

Actually we will show that, (up to equivalence), there are only
three limit functions.

Let us introduce the following sequence of vertices $\{c_n\}$,
where $c_n:=r^nu$. Geometrically the $c_n$'s are points \textit{symmetric} with respect to $o$, more
precisely there are two paths from $c_n$ to $o$ which realize the
distance.

\begin{teo}\label{teo2}
There are infinitely many horofunctions in the graph
$\Gamma_{l^{\infty}}$. More precisely:
\begin{enumerate}
  \item the function $f_U$ corresponding to the Busemann points $\lim
  U_n$;
  \item the function $f_R$ corresponding to the Busemann points $\lim
  R_n$;
  \item infinitely many functions equivalent to the function $f_c$
  obtained as $\lim f_n$ associated with the sequence $\{c_n\}$.
\end{enumerate}
\end{teo}
\begin{proof}
It is clear that $f_U$ and $f_R$ are non-equivalent horofunctions.
The function $f_c$ is not equivalent to $f_U$ and $f_R$. In fact
$f_c(U_n)=f_c(R_n)=2^{n-1}$, as one can easily check.

Moreover, all the sequences of points $\{x_n\}$ satisfying conditions of
Lemma (\ref{lemma2}) at bounded distance from $\{c_n\}$ (i.e., there
exist $M>0$ so that for every $n$ there exists $k$ such that
$d(x_n,c_k)<M$) give rise to horofunctions not equal but
equivalent to $f_c$. It remains to prove that if $\{x_n\}$ is a sequence
whose limit $f$ exists, and such that it is not at bounded distance
from $\{c_n\}$, then either $f=f_U$ or $f=f_R$. From Lemma
(\ref{lemma2}) we can suppose that, for $n$ sufficiently large, $x_n$ is a vertex in the upper
part of $\Gamma_w^k$, for some $k=k(n)$ (the case in which $x_n$ definitively belongs to the right part of $\Gamma_w^k$ is symmetric and left to the reader). Denote by $\gamma_n:=\min_k
d(x_n, c_k)$, by our assumption $\gamma_n\rightarrow\infty$. Fix a
vertex $y$, then there exists $h$ to be the minimum index such that $y\in \Gamma_w^h$.
We want to prove that $f(y)=f_U(y)$. We use the notation
$U_{k(n)}$ and $R_{k(n)}$ as before. We have that
$$
d(x_n,y)=
  \begin{cases}
    d(x_n,U_{k(n)-1})+d(U_{k(n)-1},y) & \text{or}, \\
    d(x_n, c_{k(n)})+d(c_{k(n)},R_{k(n)-1})+d(R_{k(n)-1},y) & .
  \end{cases}
$$
In order to prove that, for $n$ large enough the distance is given by the first of the two expressions above, we observe that 
$$d(x_n,U_{k(n)-1})\leq
d(c_{k(n)},R_{k(n)-1})=2^{k(n)-1},\ d(U_{k(n)-1},y)\leq
2^{k(n)-1}
$$ 
and $d(R_{k(n)-1},y)\geq 2^{k(n)-1}-2^h$. This gives
$$
2^{k(n)-1}<\gamma_n+2^{k(n)-1}-2^h \Leftrightarrow \ 2^h<\gamma_n,
$$
which is verified for $n$ large enough. So in the limit
\begin{eqnarray*}
f(y)&=& \lim_n(d(x_n, o)-d(x_n, y)) \\
&=& \lim_n(d(x_n, U_{k(n)-1})+ d(
U_{k(n)-1},o)-d(x_n,U_{k(n)-1})-d(
U_{k(n)-1},y))\\
&=& f_U(y).
\end{eqnarray*}
\end{proof}

\begin{os}\rm
The Busemann function $f_U=\lim f_n$ associated with the point
$U=\lim U_n$ can be easily described in the following way. Project
each vertex $y$ of the graph on the geodesic ray connecting $o$ to
$U$, denote by $y_U$ the image of the projection. then
$f_U(y)=d(o,y_U)$. This follows from the fact that for each $y$
the value $f_U(y)$ can be computed as difference of distances in a
finite graph $\Gamma^n_w$.

The same can be said for $f_R$.
\end{os}
\begin{center}
\begin{picture}(500,280)
\letvertex A=(180,210)\letvertex B=(150,160)\letvertex C=(120,110)

\letvertex D=(90,60)\letvertex E=(60,10)\letvertex F=(120,10)\letvertex G=(180,10)

\letvertex H=(240,10)\letvertex I=(300,10)
\letvertex L=(270,60)\letvertex M=(240,110)\letvertex N=(210,160)

\letvertex O=(180,110)\letvertex P=(150,60)\letvertex Q=(210,60)
\letvertex R=(250,210) \letvertex S=(360,10)
\letvertex T=(210,250)\letvertex U=(330,60)

\put(210,270){$f_U$} \put(370,10){$f_R$}

\drawvertex(A){$\bullet$}\drawvertex(B){$\bullet$}
\drawvertex(C){$\bullet$}\drawvertex(D){$\bullet$}
\drawvertex(E){$\ast$}\drawvertex(F){$\bullet$}
\drawvertex(G){$\bullet$}\drawvertex(H){$\bullet$}
\drawvertex(I){$\bullet$}\drawvertex(L){$\bullet$}\drawvertex(M){$\bullet$}
\drawvertex(N){$\bullet$}\drawvertex(O){$\bullet$}
\drawvertex(P){$\bullet$}\drawvertex(Q){$\bullet$}

\drawedge(A,T){}\drawundirectededge(A,R){}\drawedge(I,S){}\drawundirectededge(I,U){}
\drawundirectededge(E,D){} \drawundirectededge(D,C){}
\drawundirectededge(C,B){} \drawundirectededge(B,A){}
\drawundirectededge(A,N){} \drawundirectededge(N,M){}
\drawundirectededge(M,L){} \drawundirectededge(L,I){}
\drawundirectededge(I,H){} \drawundirectededge(H,G){}
\drawundirectededge(G,F){} \drawundirectededge(F,E){}
\drawundirectededge(N,B){} \drawundirectededge(O,C){}
\drawundirectededge(M,O){} \drawundirectededge(P,D){}
\drawundirectededge(L,Q){} \drawundirectededge(B,O){}
\drawundirectededge(O,N){} \drawundirectededge(C,P){}
\drawundirectededge(P,G){} \drawundirectededge(D,F){}
\drawundirectededge(Q,M){}
\drawundirectededge(G,Q){}\drawundirectededge(H,L){}\drawundirectededge(F,P){}
\drawundirectededge(Q,H){}

\letvertex Y=(310,150)\letvertex
Z=(340,180)\drawedge(Y,Z){}\put(350,200){$f_c$}\drawvertex(Y){$\bullet$}
\put(300,135){$c_n$}

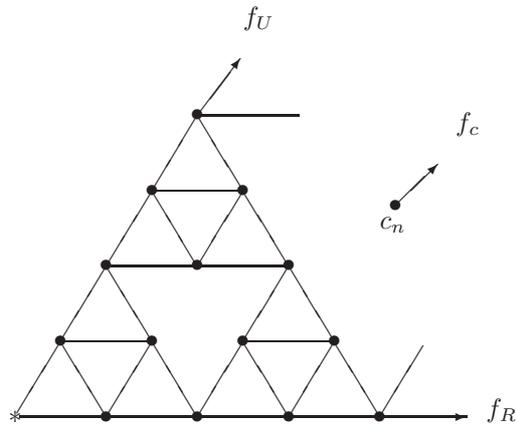
\put(60,-20){\textbf{Fig.} Horofunctions in
$\Gamma_{l^{\infty}}$ (up to equivalence)}
\end{picture}
\end{center}
\vspace{1cm}

\section*{Acknowledgments}
Some of the results contained in this paper were obtained during my staying at Technion University of Haifa. Among the others, I want to thank Uri Bader, Uri Onn, Amos Nevo, Michael Brandenbursky and Vladimir Finkelshtein for useful discussions. Moreover I am grateful to the anonymous referees for recommending various improvements in exposition.\\


\end{document}